\documentclass{amsart}

\usepackage{amsmath}
\usepackage{latexsym,amsthm,amssymb,amscd,url,enumerate}
\usepackage{bm}
\newdimen\proofrulebreadth \proofrulebreadth=.05em
\newdimen\proofdotseparation \proofdotseparation=1.25ex
\newdimen\proofrulebaseline \proofrulebaseline=2ex
\newcount\proofdotnumber \proofdotnumber=3
\let\then\relax
\def\hfi{\hskip0pt plus.0001fil}
\mathchardef\squigto="3A3B
\newif\ifinsideprooftree\insideprooftreefalse
\newif\ifonleftofproofrule\onleftofproofrulefalse
\newif\ifproofdots\proofdotsfalse
\newif\ifdoubleproof\doubleprooffalse
\let\wereinproofbit\relax
\newdimen\shortenproofleft
\newdimen\shortenproofright
\newdimen\proofbelowshift
\newbox\proofabove
\newbox\proofbelow
\newbox\proofrulename
\def\shiftproofbelow{\let\next\relax\afterassignment\setshiftproofbelow\dimen0 }
\def\shiftproofbelowneg{\def\next{\multiply\dimen0 by-1 }%
\afterassignment\setshiftproofbelow\dimen0 }
\def\setshiftproofbelow{\next\proofbelowshift=\dimen0 }
\def\setproofrulebreadth{\proofrulebreadth}

\def\prooftree{
\ifnum  \lastpenalty=1
\then   \unpenalty
\else   \onleftofproofrulefalse
\fi
\ifonleftofproofrule
\else   \ifinsideprooftree
        \then   \hskip.5em plus1fil
        \fi
\fi
\bgroup
\setbox\proofbelow=\hbox{}\setbox\proofrulename=\hbox{}%
\let\justifies\proofover\let\leadsto\proofoverdots\let\Justifies\proofoverdbl
\let\using\proofusing\let\[\prooftree
\ifinsideprooftree\let\]\endprooftree\fi
\proofdotsfalse\doubleprooffalse
\let\thickness\setproofrulebreadth
\let\shiftright\shiftproofbelow \let\shift\shiftproofbelow
\let\shiftleft\shiftproofbelowneg
\let\ifwasinsideprooftree\ifinsideprooftree
\insideprooftreetrue
\setbox\proofabove=\hbox\bgroup$\displaystyle 
\let\wereinproofbit\prooftree
\shortenproofleft=0pt \shortenproofright=0pt \proofbelowshift=0pt
\onleftofproofruletrue\penalty1
}

\def\eproofbit{
\ifx    \wereinproofbit\prooftree
\then   \ifcase \lastpenalty
        \then   \shortenproofright=0pt  
        \or     \unpenalty\hfil         
        \or     \unpenalty\unskip       
        \else   \shortenproofright=0pt  
        \fi
\fi
\global\dimen0=\shortenproofleft
\global\dimen1=\shortenproofright
\global\dimen2=\proofrulebreadth
\global\dimen3=\proofbelowshift
\global\dimen4=\proofdotseparation
\global\count255=\proofdotnumber
$\egroup  
\shortenproofleft=\dimen0
\shortenproofright=\dimen1
\proofrulebreadth=\dimen2
\proofbelowshift=\dimen3
\proofdotseparation=\dimen4
\proofdotnumber=\count255
}

\def\proofover{
\eproofbit 
\setbox\proofbelow=\hbox\bgroup 
\let\wereinproofbit\proofover
$\displaystyle
}%
\def\proofoverdbl{
\eproofbit 
\doubleprooftrue
\setbox\proofbelow=\hbox\bgroup 
\let\wereinproofbit\proofoverdbl
$\displaystyle
}%
\def\proofoverdots{
\eproofbit 
\proofdotstrue
\setbox\proofbelow=\hbox\bgroup 
\let\wereinproofbit\proofoverdots
$\displaystyle
}%
\def\proofusing{
\eproofbit 
\setbox\proofrulename=\hbox\bgroup 
\let\wereinproofbit\proofusing
\kern0.3em$
}

\def\endprooftree{
\eproofbit 
  \dimen5 =0pt
\dimen0=\wd\proofabove \advance\dimen0-\shortenproofleft
\advance\dimen0-\shortenproofright
\dimen1=.5\dimen0 \advance\dimen1-.5\wd\proofbelow
\dimen4=\dimen1
\advance\dimen1\proofbelowshift \advance\dimen4-\proofbelowshift
\ifdim  \dimen1<0pt
\then   \advance\shortenproofleft\dimen1
        \advance\dimen0-\dimen1
        \dimen1=0pt
        \ifdim  \shortenproofleft<0pt
        \then   \setbox\proofabove=\hbox{%
                        \kern-\shortenproofleft\unhbox\proofabove}%
                \shortenproofleft=0pt
        \fi
\fi
\ifdim  \dimen4<0pt
\then   \advance\shortenproofright\dimen4
        \advance\dimen0-\dimen4
        \dimen4=0pt
\fi
\ifdim  \shortenproofright<\wd\proofrulename
\then   \shortenproofright=\wd\proofrulename
\fi
\dimen2=\shortenproofleft \advance\dimen2 by\dimen1
\dimen3=\shortenproofright\advance\dimen3 by\dimen4
\ifproofdots
\then
        \dimen6=\shortenproofleft \advance\dimen6 .5\dimen0
        \setbox1=\vbox to\proofdotseparation{\vss\hbox{$\cdot$}\vss}%
        \setbox0=\hbox{%
                \advance\dimen6-.5\wd1
                \kern\dimen6
                $\vcenter to\proofdotnumber\proofdotseparation
                        {\leaders\box1\vfill}$%
                \unhbox\proofrulename}%
\else   \dimen6=\fontdimen22\the\textfont2 
        \dimen7=\dimen6
        \advance\dimen6by.5\proofrulebreadth
        \advance\dimen7by-.5\proofrulebreadth
        \setbox0=\hbox{%
                \kern\shortenproofleft
                \ifdoubleproof
                \then   \hbox to\dimen0{%
                        $\mathsurround0pt\mathord=\mkern-6mu%
                        \cleaders\hbox{$\mkern-2mu=\mkern-2mu$}\hfill
                        \mkern-6mu\mathord=$}%
                \else   \vrule height\dimen6 depth-\dimen7 width\dimen0
                \fi
                \unhbox\proofrulename}%
        \ht0=\dimen6 \dp0=-\dimen7
\fi
\let\doll\relax
\ifwasinsideprooftree
\then   \let\VBOX\vbox
\else   \ifmmode\else$\let\doll=$\fi
        \let\VBOX\vcenter
\fi
\VBOX   {\baselineskip\proofrulebaseline \lineskip.2ex
        \expandafter\lineskiplimit\ifproofdots0ex\else-0.6ex\fi
        \hbox   spread\dimen5   {\hfi\unhbox\proofabove\hfi}%
        \hbox{\box0}%
        \hbox   {\kern\dimen2 \box\proofbelow}}\doll%
\global\dimen2=\dimen2
\global\dimen3=\dimen3
\egroup 
\ifonleftofproofrule
\then   \shortenproofleft=\dimen2
\fi
\shortenproofright=\dimen3
\onleftofproofrulefalse
\ifinsideprooftree
\then   \hskip.5em plus 1fil \penalty2
\fi
}

\usepackage[all]{xy}
\SelectTips{cm}{}

\newcommand{\myemph}[1]{\textbf{#1}}    
\newcommand{\defeq}{=_{\mathrm{def}}}
\newcommand{\type}{\mathsf{type}}       

\newcommand{\nat}{\ensuremath{\mathsf{Nat}}} 
\newcommand{\Id}{\mathsf{Id}}
\newcommand{\id}[1]{\Id_{#1}}
\newcommand{\refl}{\mathsf{refl}}
\newcommand{\idrec}{\mathsf{idrec}}
\newcommand{\W}{\mathsf{W}}
\newcommand{\wsup}{\mathsf{sup}}
\newcommand{\wrec}{\mathsf{wrec}}
\newcommand{\wcomp}{\mathsf{wcomp}}
\newcommand{\Bool}{\mathsf{2}}

\newcommand{\boolrec}{\mathsf{2rec}}
\newcommand{\boolcomp}{\mathsf{2comp}}
\newcommand{\UU}{\mathsf{U}}
\newcommand{\pair}{\mathsf{pair}}
\newcommand{\Palg}{P\text{-}\mathsf{Alg}}
\newcommand{\BoolAlg}{\mathsf{2}\text{-}\mathsf{Alg}}

\newcommand{\app}{\mathsf{app}}

\newcommand{\iscontr}{\mathsf{iscontr}}
\newcommand{\hfiber}{\mathsf{hfiber}}
\newcommand{\wrecs}{\mathsf{simp\textsf{-}wrec}}
\newcommand{\Hint}{\mathcal{H}}
\newcommand{\Hext}{\mathcal{H}_{\mathrm{ext}}}

\newtheorem{theorem}{Theorem}
\newtheorem*{theorem*}{Theorem}
\newtheorem{proposition}[theorem]{Proposition}

\theoremstyle{remark}
\newtheorem{remark}[theorem]{Remark} 
\newtheorem*{remarks*}{Remarks}

\theoremstyle{definition}
\newtheorem{definition}[theorem]{Definition}

\begin{document}

\title{Inductive Types in Homotopy Type Theory}

\author{Steve Awodey}
\address{Department of Philosophy, Carnegie Mellon University}
\email{awodey@cmu.edu}

\author{Nicola Gambino}
\address{Dipartimento di Matematica e Informatica, 
Universit\`a degli Studi di Palermo}
\email{ngambino@math.unipa.it}

\author{Kristina Sojakova}
\address{School of Computer Science,
Carnegie Mellon University}
\email{kristinas@cmu.edu}

\date{May 2nd, 2012}

\maketitle

\begin{abstract}
\noindent Homotopy type theory is an interpretation of Martin-L\"of's constructive type theory into abstract homotopy theory.   There results a link between constructive mathematics and algebraic topology, providing topological semantics for intensional systems of type theory as well as a computational approach to algebraic topology via type theory-based proof assistants such as~Coq.

The present work investigates inductive types in this setting. Modified rules for inductive types, including types of well-founded trees, or W-types, are presented, and the basic homotopical semantics of such types are determined.  Proofs of all results have been formally verified by the Coq proof assistant, and the proof scripts for this verification form an essential component of this research.

\end{abstract}
%

\section*{Introduction}

\noindent 
The constructive type theories introduced by Martin-L\"of are dependently-typed
$\lambda$-calculi with operations for identity types $\id{A}(a,b)$, dependent products $(\Pi x {\, : \,}  A)B(x)$
and dependent sums $(\Sigma x {\, : \,} A)B(x)$,  among others~\cite{MartinLofP:intttp,MartinLofP:conmcp,MartinLofP:inttt,NordstromB:promlt,NordstromB:marltt}.   
These are related to the basic concepts of predicate logic, \emph{viz.}
equality and quantification, via the familiar propositions-as-types correspondence~\cite{HowardWH:foratn}. The different systems introduced by Martin-L\"of over the years 
vary greatly both in proof-theoretic strength~\cite{GrifforE:strsml} and computational
properties. From the computational point of view, it is important to
distinguish between the extensional systems, that have a stronger
notion of equality, but for which type-checking is undecidable, and the intensional ones, that have a
weaker notion of equality, but for which type-checking is decidable~\cite{HofmannM:extcit,MaiettiME:mintlf}. For example, the type theory presented in~\cite{MartinLofP:inttt} is extensional, while that in~\cite{NordstromB:marltt} is 
intensional.

The difference between the extensional and the intensional treatment of equality has a strong 
impact also on the properties of the various types that may be assumed in a type theory, and in particular
on those of inductive types, such as the types of Booleans, natural numbers, lists and W-types~\cite{MartinLofP:conmcp}.  Within extensional type theories, inductive types can be 
characterized (up to isomorphism) as initial algebras of certain definable functors. The initiality condition 
translates directly into a recursion principle that expresses the existence and uniqueness of 
recursively-defined functions. In particular, W-types can be characterized as initial algebras of polynomial functors~\cite{DybjerP:repids,MoerdijkI:weltc}. Furthermore, within extensional type theories, W-types allow us to define a wide range of inductive types, such as the type of natural numbers and types of lists~\cite{DybjerP:repids,GambinoN:weltdp,AbbottM:concsp}.
Within intensional type theories, by contrast, the correspondence between inductive types and initial algebras 
breaks down, since it is not possible to prove the uniqueness of recursively-defined functions. 
Furthermore, the reduction of inductive types like the natural numbers to W-types fails~\cite{DybjerP:repids,GoguenH:inddtw}.

In the present work, we exploit insights derived from the new models of intensional type theory based on 
homotopy-theoretic ideas~\cite{AwodeyS:homtmi,VoevodskyV:notts,vandenBergB:topsmi} to 
investigate inductive types, thus contributing to the new area known as Homotopy Type Theory.
Homotopical intuition justifies the assumption of a limited form of function extensionality, which, 
as we show, suffices to deduce uniqueness properties of recursively-defined functions up to  
homotopy.
Building on this observation, we introduce the notions of \emph{weak algebra homomorphism} and \emph
{homotopy-initial algebra}, which require uniqueness of homomorphisms up to homotopy. We modify the rules for W-types by replacing the definitional equality in the standard
computation rule with its propositional counterpart, yielding a weak form of the corresponding inductive type. 
Our main result is that these new, weak W-types correspond precisely to homotopy-initial algebras of 
polynomial functors.   Furthermore, we indicate how homotopical versions of various inductive types can be defined as special cases of the general construction in the new setting

The work presented here is motivated in part by the Univalent Foundations program formulated by 
Voevodsky~\cite{VoevodskyV:unifp}.  This ambitious program intends to provide comprehensive foundations for mathematics on the basis of homotopically-motivated type theories, with an associated computational implementation in the Coq proof assistant.  The present investigation of inductive types   serves as an example of this new paradigm: despite the fact that the intuitive basis lies in higher-dimensional category theory and homotopy theory, the actual development is strictly syntactic, allowing for direct formalization in Coq.  Proof scripts of the definitions, results, and all necessary preliminaries are provided in a downloadable repository~\cite{AwodeyS:indtht}. 

The paper is organized as follows.  In section \ref{section:prelims}, we describe and motivate the dependent type theory over which we will work and compare it to some other well-known systems in the literature. The basic properties of the system and its homotopical interpretation are developed to the extent required for the present purposes. Section \ref{section:extW} reviews the basic theory of W-types in extensional type theory and sketches the proof that these correspond to initial algebras of polynomial functors; there is nothing new in this section, rather it serves as a framework for the generalization that follows.  Section \ref{section:extW} on intensional W-types contains the development of our new theory; it begins with a simple example, that of the type $\Bool$ of Boolean truth values, which serves to indicate the main issues involved with inductive types in the intensional setting, and our proposed solution.  We then give the general notion of weak W-types, including the crucial new notion of \emph{homotopy-initiality}, and state our main result, the equivalence between the type-theoretic rules for weak W-types and the existence of a homotopy-initial algebra of the corresponding polynomial functor.  
Moreover, we show how some of the difficulties with intensional W-types are remedied in the new setting by showing that the type of natural numbers can be defined as an appropriate W-type.
Finally, we conclude by indicating how this work fits into the larger study of inductive types in Homotopy Type Theory and the Univalent Foundations program generally.

\section{Preliminaries}\label{section:prelims}

\noindent 
The general topic of Homotopy Type Theory is concerned with the study of the constructive type theories of Martin-L\"of under their new interpretation into abstract homotopy theory and higher-dimensional category theory. Martin-L\"of type theories are foundational systems which have been used to formalize large parts of constructive mathematics, and also for the development of high-level programming languages~\cite{MartinLofP:conmcp}.  They are prized for their combination of expressive strength and desirable proof-theoretic properties.  One aspect of these type theories that has led to special difficulties in providing semantics is the intensional character of equality.  In recent work \cite{AwodeyS:homtmi,VoevodskyV:notts,vandenBergB:topsmi,AwodeyS:typth}, it has emerged that the topological notion of \emph{homotopy} provides an adequate basis for the semantics of intensionality.  This extends the paradigm of computability as continuity, familiar from domain theory, beyond the simply-typed 
$\lambda$-calculus to dependently-typed theories involving:\begin{enumerate}[(i)]
\item dependent sums $(\Sigma x\colon\!{A})B(x)$ and dependent products $(\Pi x\colon\!{A})B(x)$, modelled respectively by the total space and the space of sections of the fibration modelling the dependency of $B(x)$ over $ x : A$; \item
and, crucially, including the identity type constructor~$\id{A}(a,b)$, interpreted as the space of all \emph{paths} in~$A$ between points~$a$ and~$b$. \end{enumerate}

In the present work, we build on this homotopical interpretation to study inductive types, such as the natural numbers, Booleans, lists, and W-types. Within extensional type theories, W-types can be used to  provide a constructive counterpart of the classical notion of a well-ordering~\cite{MartinLofP:inttt} and to uniformly define a variety of inductive types~\cite{DybjerP:repids}.
However, most programming languages and proof assistants, such as Coq~\cite{BertotY:inttpp}, Agda~\cite{NorellU:towppl} and Epigram~\cite{McBrideC:viefl} use schematic inductive definitions~\cite{CoquandT:inddt,PaulinMorhringC:inddsc} rather than W-types to define inductive types.  This is due in part to the practical convenience of the schematic approach, but it is also a matter of necessity; these systems are based on intensional rather than extensional type theories, and in the intensional theory the usual reductions of inductive types to W-types fail~\cite{DybjerP:repids,McBrideC:wtygnb}.
Nonetheless, W-types retain great importance from a theoretical perspective, since they allow us to internalize in type theory arguments about inductive types. Furthermore, as we will see in Section~\ref{section:intW}, a limited form of extensionality licensed by the homotopical interpretation suffices to develop the theory of W-types in a satisfactory way. In particular, we shall make use of ideas from higher category theory and homotopy theory to understand W-types as ``homotopy-initial" algebras of an appropriate kind.

\subsection{Extensional vs.\ intensional type theories}

\noindent We work here with type theories that have the four standard forms of judgement
\[
A : \type \, , \quad A = B : \type \, , \quad   a : A \, , \quad a = b : A \, . 
\]
We refer to the equality relation in these judgements as \emph{definitional equality}, 
which should be contrasted with the notion of \emph{propositional equality}
recalled below. 
Such a judgement $J$ can be made also relative to a \emph{context}~$\Gamma$ of variable declarations, a situation that we indicate by writing~$\Gamma \vdash J$. When stating deduction
rules we make use of standard conventions to simplify the exposition, such as omitting the mention
of a context that is common to premisses and conclusions of the rule.
The rules for identity types in intensional type theories are given in~\cite[Section~5.5]{NordstromB:marltt}. We recall them here in a slighly different, but equivalent, formulation.

\begin{itemize}
\item $\Id$-formation rule.
\[
\begin{prooftree}
A :  \type \quad 
a :  A  \quad
b :  A 
\justifies
 \id{A}(a,b) :  \type
 \end{prooftree}
\]
\item $\Id$-introduction rule.
\[
\begin{prooftree}
a :  A 
\justifies
 \refl(a) :  \id{A}(a,a)
 \end{prooftree} 
\]
\item $\Id$-elimination rule.
\[
\begin{prooftree}
\begin{array}{l} 
x, y :  A, u :  \id{A}(x,y) \vdash C(x,y,u) :  \type \\
 x :  A \vdash  c(x) :  C(x,x,\refl(x))  
 \end{array}
\justifies
x, y :  A, u :  \id{A}(x,y) \vdash  \idrec(x,y,u,c) :  C(x,y,u)
\end{prooftree}
\]
\item $\Id$-computation rule.
\[
\begin{prooftree}
\begin{array}{l} 
x, y :  A, u :  \id{A}(x,y) \vdash C(x,y,u) :  \type \\
 x :  A \vdash  c(x) :  C(x,x,\refl(x)) 
 \end{array}
 \justifies
x :  A \vdash \idrec(x,x,\refl(x), c) = c(x) :  C(x, x, \refl(x)) \, .
\end{prooftree}
\]
\end{itemize}

\medskip

As usual, we say that two elements  $a, b :A$ are \emph{propositionally equal} if 
 the type $\Id(a,b)$ is inhabited.
Most work on W-types to date (\emph{e.g.}~\cite{DybjerP:repids,MoerdijkI:weltc,AbbottM:concsp}) has been in the setting of extensional type theories,  
in which the following rule, known as the \emph{identity reflection rule}, is also assumed:

\begin{equation}
\label{equ:collapse}
\begin{prooftree}
 p :  \id{A}(a,b)
  \justifies
  a=b :  A
\end{prooftree}
\end{equation}

This rule collapses propositional equality with definitional equality, thus making the overall system
somewhat simpler to work with. However, it destroys the constructive character of the intensional system, since it makes type-checking undecidable~\cite{HofmannM:extcit}. For this reason, it is not assumed
in the most recent formulations of Martin-L\"of type theories~\cite{NordstromB:marltt} or in automated proof assistants like Coq~\cite{BertotY:inttpp}.

In intensional type theories, inductive types cannot be characterized by standard category-theoretic
universal properties. For instance, in this setting it is not possible to show that there exists a 
definitionally-unique function out of the empty type with rules as in~\cite[Section~5.2]{NordstromB:marltt}, thus making it impossible to prove that the empty type provides an initial object. 
Another consequence of this fact is that, if we attempt to define the type of 
natural numbers as a W-type in the usual way, then 
the usual elimination and computation rules for it are no longer derivable~\cite{DybjerP:repids}. Similarly, it is not possible to show the uniqueness of recursively-defined functions out of W-types. When interpreted categorically, the uniqueness of such functions translates into the initiality property of the associated polynomial functor algebra, which is why the correspondence between W-types and initial algebras fails in the intensional setting.

Due to this sort of poor behaviour of W-types, and other constructions, in the purely intensional setting, that system is often augmented by other extensionality principles that are somewhat weaker than the Reflection rule, such as Streicher's K-rule  or the Uniqueness of Identity Proofs (UIP)~\cite{StreicherT:invitt}, which has recently been reconsidered
in the context of Observational Type Theory \cite{AltenkirchT:obsen}.  Inductive types in such intermediate systems are somewhat better behaved, but still exhibit some undesirable properties, making them less useful for practical purposes than one might wish~\cite{McBrideC:wtygnb}.  Moreover, these intermediate systems seem to lack a clear conceptual basis:  they neither intend to formalize constructive sets (like the extensional theory) nor is there a principled reason to choose these particular extensionality rules, beyond their practical advantages.  

\subsection{The system $\Hint$} 

\noindent We here take a different approach to inductive types in the intensional setting, namely, one motivated by the homotopical interpretation.  It involves working over a dependent type theory $\Hint$ which has the following deduction rules on top of the standard structural rules:
\begin{itemize}
\item rules for identity types as stated above;
\item rules for $\Sigma$-types as in~\cite[Section~5.8]{NordstromB:marltt};
\item rules for $\Pi$-types as in~\cite[Section~3.2]{GarnerR:strdpt}; 
\item the propositional $\eta$-rule for $\Pi$-types, \emph{i.e.} the axiom
asserting that for every $f : (\Pi x : A) B(x)$, the type $\Id(f, 
\lambda x. \app(f,x))$ is inhabited;
\item the Function Extensionality axiom (FE), \emph{i.e.} the axiom asserting that
for every $f, g : A \rightarrow B$, the type
\[
(\Pi x :  A)\id{B}( \app(f, x), \app(g, x)) \rightarrow \id{A \rightarrow B}(f,g) 
\]
is inhabited.
\end{itemize}
Here, we have used the notation $A \rightarrow B$ to indicate function types, defined via
$\Pi$-types in the usual way. Similarly, we will write $A \times B$ to denote the binary product
of two types as usually defined via $\Sigma$-types.
\smallskip

\subsubsection*{Remarks}
\begin{enumerate}[(i)]
\item The rules for $\Pi$-types of $\Hint$ are derivable from those
in~\cite[Section~5.4]{NordstromB:marltt}. For simplicity, 
we will write~$f(a)$ or~$f  a$ instead of $\app(f,a)$. 
\item As shown in~\cite{VoevodskyV:unifc}, the $\eta$-rule for dependent
functions and the function extensionality principle stated above imply
the corresponding function extensionality principle for dependent functions, \emph{i.e.} 
\[
(\Pi x :  A)\id{B(x)}( f x, g x) \rightarrow \id{(\Pi x : A) B(x)}(f,g) \, .
\]
\item The following form of the $\eta$-rule for $\Sigma$-types is derivable:
\[
\begin{prooftree}
c  : (\Sigma x : A)B(x) 
\justifies
\eta_{\Sigma}(c) : \Id(c, \pair( \pi_1 c \, , \pi_2 c)) \, , 
\end{prooftree}
\]
 where $\pi_1$ and $\pi_2$ are the projections. This  can be proved by $\Sigma$-elimination,
without FE.
\item $\Hint$ does \emph{not} include the $\eta$-rules as definitional equalities, either for $\Sigma$-types or for $\Pi$-types (as is done in~\cite{GoguenH:inddtw}).
\item The type theory $\Hint$ will serve as the background theory for our study of 
inductive types and W-types. For this reason, we need not assume it to have any primitive types.
\end{enumerate}

\noindent
This particular combination of rules is motivated by the fact that $\Hint$ has a clear
homotopy-theoretic sematics. Indeed, the type theory~$\Hint$ is a subsystem of the type theory 
used in Voevodsky's Univalent Foundations library~\cite{VoevodskyV:unifc}.  In particular, the 
Function Extensionality axiom is formally implied by Voevodsky's Univalence axiom~\cite{VoevodskyV:notts}, 
which is also valid in homotopy-theoretic models, but will not be needed here. Note that, 
while the Function Extensionality axiom is valid also in set-theoretic models, the Univalence 
axiom is not. Although $\Hint$ has a straightforward set-theoretical semantics, we stress that it 
does not have any global extensionality rules, like the identity reflection rule, K, or UIP. This makes it also compatible with ``higher-dimensional" interpretations such as the groupoid model~\cite{HofmannM:gromtt}, in which the rules of $\Hint$ are also valid.

\subsection{Homotopical semantics} 

\noindent The homotopical semantics of  $\Hint$ is based on the idea that an identity term~$p:  \id{A}(a,b)$ 
is (interpreted as) a path $p: a\leadsto b$ between the points $a$ and $b$ in the space $A$.   
More generally, the interpretations of terms $a(x)$ and $b(x)$ with free variables will be continuous 
functions into the 
space $A$, and an identity term $p(x) :  \id{A}\big(a(x),b(x)\big)$ is then a 
continuous family of paths, \emph{i.e.}~a homotopy between the continuous functions. Now, the main import of the 
$\Id$-elimination rule is that  type dependency must respect identity, in the following sense: given a dependent type
\begin{equation}
\label{equ:deptype}
x:A \vdash B(x) : \type \, ,
\end{equation} 
and $p: \id{A}(a,b)$, there is then a \emph{transport} function 
 $$p_{\, ! } : B(a) \rightarrow B(b),$$ which is defined by $\Id$-elimination, taking for $x : A$
the function $\refl(x)_{\, !} : B(x) \rightarrow B(x)$ to be the identity on $B(x)$.  Semantically, 
given that an identity term $p: \id{A}(a,b)$ is interpreted as a path $p: a\leadsto b$, 
 this means that a dependent type as in~\eqref{equ:deptype} must be interpreted as a space $B\rightarrow A$, fibered
 over the space $A$,  and that the judgement
  \[
  x,y:A \vdash\id{A}(x,y) : \type
  \] 
  is interpreted as the canonical fibration $A^I \rightarrow A\times A$ 
 of the path space $A^I$ over $A \times A$. For a more detailed overview of the homotopical interpretation, 
 see~\cite{AwodeyS:typth}.

Independently of this interpretation, each type $A$ can be shown to carry the structure of a weak 
$\omega$-groupoid in the sense of~\cite{BataninM:mongcn,LeinsterT:higohc} with the elements of $A$ as objects, identity proofs $p : \id{A}(a,b)$ as morphisms and 
 elements of iterated identity types 
 as~$n$-cells~\cite{vandenBergB:typwg,LumsdaineP:weaci}. Furthermore, $\Hint$ 
 determines a weak $\omega$-category~$\mathcal{C}(\Hint)$ having types as 0-cells, elements $f : A \rightarrow B$ as 1-cells, and elements of (iterated) identity types 
as~$n$-cells~\cite{Lumsdaine:higcft}.   The relation between the weak $\omega$-category structure of~$\mathcal{C}(\Hint)$ and the homotopical interpretation of intensional type theories closely mirrors that between higher category theory and homotopy theory in modern algebraic topology, and some methods developed in the latter setting are also applicable in type theory.  For instance,
 the topological notion of contractibility admits the following type-theoretic counterpart, originally
 introduced by Voevodsky in~\cite{VoevodskyV:unifc}.

\begin{definition}  A type $A$ is called \emph{contractible} if the  type 
 \begin{equation}
 \label{eq:contractible}
\iscontr(A) \defeq (\Sigma x:A)(\Pi y:A)\id{A}(x,y)
\end{equation}
is inhabited.
\end{definition} 

The type $\iscontr(A)$ can be seen as the propositions-as-types translation
of the formula stating that $A$ has a unique element. However, its homotopical interpretation 
is as a space that is inhabited if and only if the space interpreting $A$ is contractible in the usual
topological sense. The notion of contractibility can be used to articulate the world of types into different homotopical dimensions, or \emph{h-levels} \cite{VoevodskyV:unifc}. This classification has proven to be quite useful in understanding intensional type theory.  
For example, it permits the definition of new notions of \emph{proposition} and \emph{set} which provide a useful alternative to the standard approach to formalization of mathematics in type theory~\cite{VoevodskyV:unifc}.

\begin{remark} \label{thm:idcontrcontr}
If $A$ is a contractible type, then for every $a, b : A$, the type $\id{A}(a,b)$ is again contractible. This can be proved  by $\Id$-elimination~\cite{AwodeyS:indtht}. 
\end{remark}

Let us also recall from~\cite{VoevodskyV:unifc} the notions of weak equivalence and homotopy equivalence. To do this, we need to fix some notation. For $f : A \rightarrow B$
and $y : B$, define the type
\[
 \hfiber(f,y) \defeq (\Sigma x : A) \id{B}(f x, y) \, .
\]
We refer to this type as the \emph{homotopy fiber} of $f$ at $y$. 

\begin{definition} \label{thm:weq} Let $f : A \rightarrow B$.
\begin{itemize}
\item We say that $f$ is a \emph{weak equivalence} if  the type
\[
\mathsf{isweq(f)} \defeq (\Pi y : B) \,  \iscontr(\hfiber(f,y)) 
\]
is inhabited. 
\item We say that $f$ is a \emph{homotopy equivalence} if there exist a function 
$g : B\rightarrow A$ and elements
\begin{align*}
\eta &: (\Pi x : A) \Id( g  f  x , x) \,  ,\\
\varepsilon &: (\Pi y:B) \Id( f   g  y, y) \, .
\end{align*}
It is an \emph{adjoint homotopy equivalence} if there are also
terms
\begin{align*}
p &: (\Pi x : A) \Id ( \varepsilon_{f x} \, , f \, \eta_x )  \, , \\
q &: (\Pi y : B) \Id ( \eta_{g y} \, , g \, \varepsilon_y) \, ,
\end{align*}
where the same notation for both function application and
the action of a function on an identity proof (which is easily definable by $\Id$-elimination),
and we write $\alpha_x$ instead of $\alpha(x)$ for better readability.
\end{itemize}
\end{definition}

The type $\mathsf{isweq(f)}$ can be seen as the propositions-as-types translation of the formula asserting that $f$ is bijective, while homotopy equivalence is evidently a form of isomorphism. Thus it is a pleasant fact that a function is a weak equivalence if and only if it is a homotopy equivalence~\cite{VoevodskyV:unifc}. 
We also note that all type-theoretic constructions are homotopy invariant, in the sense that they respect this relation of equivalence, a fact which is exploited by the Univalence axiom~\cite{VoevodskyV:notts}.

\medskip

In Section \ref{section:intW} below, these and related homotopy-theoretic insights will be used to study inductive types, but first we must briefly review some basic facts about  inductive types in the extensional setting.

\section{Extensional W-types}\label{section:extW}

\noindent 
We briefly recall the theory of W-types in fully extensional type theories. Let us begin by recalling
the rules for W-types from~\cite{MartinLofP:inttt}. To state them more conveniently, we sometimes
write~$W$ instead of $(\W x : A) B(x)$. 

\smallskip

\begin{itemize}
\item $\W$-formation rule.
\[
\begin{prooftree}
 A : \type \qquad
 x : A \vdash B(x) : \type
 \justifies
 (\W x : A) B(x) : \type
\end{prooftree}
\]
\item $\W$-introduction rule.
\[
\begin{prooftree}
a:A \qquad
t : B(a) \rightarrow W
\justifies
\wsup(a, t): W
\end{prooftree}
\]
\item $\W$-elimination rule.
\[
 \begin{prooftree}
 \hspace{-1ex} \begin{array}{l}
 w : W \vdash C(w) : \type \\ 
 x:A \, , u:  B(x) \rightarrow W \, ,  v : (\Pi y : B(x)) C(u(y))  \vdash \\ 
   \qquad c(x,u,v) : C(\wsup(x,u))
   \end{array}
 \justifies
w : W \vdash    \wrec(w,c) : C(w) 
\end{prooftree}
\]
\item $\W$-computation rule.
\[
 \begin{prooftree}
 \hspace{-1ex} \begin{array}{l}
 w : W \vdash C(w) : \type \\ 
 x:A \, , u:  B(x) \rightarrow W \, ,  v : (\Pi y : B(x)) C(u(y))  \vdash \\ 
   \qquad c(x,u,v) : C(\wsup(x,u))
   \end{array}
 \justifies
 \begin{array}{l} 
x : A, u : B(x) \rightarrow W \vdash 
\wrec(\wsup(x,u), c) = \\ 
  \qquad c(x,u, \lambda y. \wrec(u(y), c)) : C(\wsup(x,u)) \, .
\end{array}
\end{prooftree}
\]
\end{itemize}

\noindent

\medskip

W-types can be seen informally as the free algebras for signatures
with operations of possibly infinite arity, but no equations. Indeed, the premisses 
of the formation rule above can be thought of as specifying a signature that has the elements of~$A$ 
as operations and in which the arity of~ $a : A$ is the cardinality of the type~$B(a)$. Then, the introduction rule specifies the canonical way of forming an element of the free algebra, and the elimination rule can be seen as the propositions-as-types translation of the appropriate induction principle.

In extensional type theories, this informal description can easily be turned into a precise
mathematical characterization. To do so, let us use the theory $\Hext$ obtained
by extending $\Hint$ with the reflection rule in~\eqref{equ:collapse}. Let $\mathcal{C}(\Hext)$ be the category with
types as objects and elements $f : A \rightarrow B$ as maps, in which two maps are
considered equal if and only if they are definitionally equal. The premisses of the introduction
rule determines the \emph{polynomial endofunctor} $P : \mathcal{C}(\Hext) \rightarrow \mathcal{C}(\Hext)$
defined by 
\[
    P(X) \defeq (\Sigma x : A ) (B(x) \rightarrow X) \, .
\]
A $P$-algebra is a pair consisting of a type $C$ and a function $s_C : PC \rightarrow C$, called 
the structure map of the algebra. The formation rule gives us an object $W \defeq (\W x : A)
B(x)$ and the introduction rule (in combination with the rules for $\Pi$-types
and $\Sigma$-types) provides a structure map
\[
s_W : PW \rightarrow W \, .
\]
The elimination rule, on the other hand, states that in order for the projection $\pi_1 \colon C \rightarrow W$, where
$C \defeq (\Sigma w {\, : \, } W) C(w)$, to have a section $s$, as in the diagram
\[
\xymatrix{
& C  \ar[d]^{\pi_1} \\
W \ar[ru]^{s} \ar[r]_{1_W} & W, 
}
\]
it is sufficient for the type $C$ to have a $P$-algebra structure over $W$. Finally, the computation rule states that the section $s$ given by the elimination rule is also a $P$-algebra homomorphism. 

The foregoing elimination rule implies what we call the \emph{simple} $\W$-elimination rule:
\[
\begin{prooftree}
C : \type  \qquad
 x : A, v : B(x) \rightarrow C \vdash c(x,v) : C
 \justifies
w : W \vdash \wrecs(w,c) :  C
\end{prooftree}
\]
This can be recognized as a recursion principle for maps from $W$ into $P$-algebras, since
the premisses of the rule describe exactly a type $C$ equipped with a structure map $s_C 
: PC  \rightarrow C$. For this special case of the elimination rule, the corresponding computation rule again states that the function
\[
\lambda w. \wrecs(w,c) : W \rightarrow C \, ,
\] 
where $c(x,v) = s_C(\pair(x,v))$ for $x : A$ and $v : B(x) \rightarrow C$, is a $P$-algebra homomorphism.
Moreover, this homomorphism can then be shown to be definitionally unique using the elimination
rule, the principle of function extensionality and the reflection rule.  The converse implication also holds: one can derive the general $\W$-elimination rule from the simple elimination rule and the following $\eta$-rule
\begin{equation*}
\begin{prooftree}
\begin{array}{l}
C : \type  \qquad w : W \vdash h(w) : C \\ 
x : A, v : B(x) \rightarrow C \vdash c(x,v) : C\\
x:A \, , u:  B(x) \rightarrow W  \vdash h\left(\wsup(x,u)) = c(x,\lambda y. hu(y)\right) : C
  \end{array}
 \justifies
w : W \vdash  h(w) =  \wrecs(w,c) :  C
\end{prooftree}
\end{equation*}
stating the uniqueness of the $\wrecs$ term among algebra maps. 
Overall, we therefore have that  in 
$\Hext$ induction and recursion are interderivable: 
\[
\begin{array}{ccc}
\text{\underline{\myemph{Induction}}} & \Leftrightarrow & \text{\underline{\myemph{Recursion}}}\\[1ex]
\text{$\W$-elimination} & & \text{Simple $\W$-elimination}\\
\text{$\W$-computation} &&  \text{Simple $\W$-computation + $\eta$-rule} 
\end{array}
\]

Finally, observe that what we are calling recursion is equivalent to the statement that the
type $W$, equipped with the structure map $s_W : PW \rightarrow W$ 
is the initial $P$-algebra. Indeed, assume the simple elimination rule, the simple computation
rule and the $\eta$-rule; then for any $P$-algebra $s_C : PC\rightarrow C$, there is a 
function $f : W \rightarrow C$ by the simple elimination rule, which is a homomorphism by the computational 
rule, and is the unique such homomorphism by the $\eta$-rule.  The converse implication from initiality to recursion is just as direct. Thus, in the extensional theory, to have an initial algebra for the endofunctor $P$ is the same thing
as having a type~$W$ satisfying the introduction, elimination and computation rules above.  Section~\ref{section:intW} will be devoted to generalizing this equivalence to the setting of Homotopy Type Theory.

\subsection{Inductive types as W-types}

\noindent To conclude our review, recall that in extensional type theory, many inductive types can be reduced to W-types.  We mention the following  examples, among many others (see \cite{MartinLofP:inttt}, \cite{DybjerP:repids}, \cite{GoguenH:inddtw}, \cite{MoerdijkI:weltc}, \cite{GambinoN:weltdp}, \cite{AbbottM:concsp}):
\begin{enumerate}
\item \emph{Natural numbers}. \label{extnatW}
The usual rules for $\nat$ as an inductive type can be derived from its formalization as the following W-type. Consider the signature with two operations, one of which has arity $0$ and one of which has arity $1$; it is presented type-theoretically by a dependent type with corresponding polynomial functor (naturally isomorphic to)
\[
P(X) = \mathsf{1} + X \, ,
\]
and the natural numbers $\nat$ together with the canonical element $0:\nat$ and the successor function $s : \nat\rightarrow\nat$ form an initial $P$-algebra
\[
(0, s) : \mathsf{1} + \nat \rightarrow \nat\, .
\]
\item \emph{Second number class.}
As shown in~\cite{MartinLofP:inttt}, the second number class can be obtained as a W-type determined by the polynomial functor 
\[
P(X) = \mathsf{1} + X + (\nat \rightarrow X) \, .
\]
This has algebras with three operations, one of arity $0$, one of arity $1$, and one of arity (the cardinality of) $\nat$.
%
\end{enumerate}

\smallskip

\section{Intensional W-types}\label{section:intW}

\noindent We begin with an example which serves to illustrate, in an especially simple case, some aspects of our theory.  The type of Boolean truth values is not a W-type, but it can be formulated as an inductive type in the familiar way by means of  formation, introduction, elimination, and computation rules.  It then has an ``up to homotopy" universal property of the same general kind as the one that we shall formulate in section \ref{subsection:main} below for W-types, albeit in a simpler form.

\subsection{Preliminary example}\label{subsection:prelimex}

\noindent The standard rules for the type $\Bool$ given in~\cite[Section~5.1]{NordstromB:marltt}
can be stated equivalently as follows.

\begin{itemize}
\item $\Bool$-formation rule.
\[
 \Bool : \type \, .
 \]
\item $\Bool$-introduction rules.
\[
0 : \Bool \, ,  \qquad  1 : \Bool \, .
\]
\item $\Bool$-elimination rule.\smallskip
\[
\begin{prooftree}
x : \Bool \vdash C(x) : \type \qquad
c_0 : C(0) \qquad
c_1 : C(1) 
\justifies
x: \Bool \vdash \boolrec(x, c_0, c_1) : C(x) 
\end{prooftree}
\]
\item $\Bool$-computation rules. 
\begin{equation*}
\begin{prooftree}
x : \Bool \vdash C(x) : \type \qquad
c_0 : C(0) \qquad
c_1 : C(1) 
\justifies
\left\{
\begin{array}{c} 
 \boolrec(0, c_0, c_1)  =  c_0 : C(0)  \, , \\
 \boolrec(1, c_0, c_1)  =  c_1 : C(1) \, .
 \end{array}
\right.
\end{prooftree}
 \end{equation*} 
\end{itemize}

Although these rules are natural ones to consider in the intensional setting, they do not imply a strict universal property. For example, given a type $C$ and elements $c_0, c_1 : 
C$, the function $\lambda x . \boolrec(x, c_0,c_1) : \Bool \rightarrow C$ cannot  be shown to be definitionally unique among the functions $f :  
\Bool \rightarrow C$ with the property that $f(0)=c_0 : C$ and $f(1)=c_1 : C$.  
The best that one can do by using $\Bool$-elimination over a suitable identity type, and function extensionality, is to show that it is unique among all such maps up to an identity term, which itself is unique up to a higher identity, which in turn is unique up to \ldots. This sort of weak $\omega$-universality, which apparently involves infinitely much data, can nonetheless be captured directly within the system of type theory (without resorting to coinduction) using ideas from higher category theory. To do so, let us define a \emph{$\Bool$-algebra} to be a type $C$ equipped with two elements
$c_0 \, , c_1 : C$. Then, a \emph{weak homomorphism} of $\Bool$-algebras
$(f, p_0, p_1) : (C, c_0,c_1)\rightarrow (D,d_0,d_1)$  
consists of a  function $f : C\rightarrow D$ together with identity terms 
\[
p_0 :  \id{D}(f (c_0) ,d_0) \, , \qquad p_1 :   \id{D}(f (c_1), d_1) \, .
\]
This is a \emph{strict homomorphism} when $f (c_0) = d_0 : D$, $f (c_1) = d_1 : D $ and the identity terms $p_0$ and $p_1$ are the corresponding reflexivity terms.  We can then define the type of weak homomorphisms from $(C, c_0,c_1)$ to $(D,d_0,d_1)$ by letting
\begin{multline*}
\BoolAlg [ (C, c_0,c_1), (D,d_0,d_1) \big] \defeq \\
 (\Sigma f: C \rightarrow D) \Id(f (c_0), d_0) \times\id{D}(f (c_1), d_1) \, .
\end{multline*}
The weak universality condition on the $\Bool$-algebra $(\Bool, 0, 1)$ that we seek can now be determined as follows. 

\begin{definition} \label{thm:boolhinitial}
A $\Bool$-algebra $(C, c_0,c_1)$ is \emph{homotopy-initial} if for any $\Bool$-algebra $(D,d_0,d_1)$, the type 
\[
\BoolAlg \big[ (C, c_0,c_1), (D,d_0,d_1)\big]
\] 
is contractible.
\end{definition}

\noindent The notion of homotopy initiality, or h-initiality for short, captures in a precise way the informal idea that there is essentially one weak algebra homomorphism $(\Bool, 0, 1) \rightarrow (C,c_0,c_1)$.  Moreover, h-initiality can be shown to follow from the rules of inference for $\Bool$ stated above.  Indeed, the  computation rules for $\Bool$ stated above
 evidently make the function
\[
\lambda x . \boolrec(x, c_0,c_1) : \Bool \rightarrow C
\]
into a \emph{strict} algebra map, a stronger condition than is required for h-initiality.  Relaxing these 
definitional equalities to propositional ones, we arrive at the following rules.

\smallskip

\noindent
\begin{itemize}
\item Propositional $\Bool$-computation rules.
\[
\begin{prooftree}
x : \Bool \vdash C(x) : \type \qquad
c_0 : C(0) \qquad
c_1 : C(1) 
\justifies
\left\{
\begin{array}{c} 
\boolcomp_0(c_0, c_1) :  \id{C(0)} \big(\boolrec(0, c_0, c_1), c_0)  \, , \\ 
\boolcomp_1(c_0, c_1)  : \id{C(1)} \big(\boolrec(1, c_0, c_1), c_1)  \, .
\end{array}
\right.
\end{prooftree}
\]
\end{itemize}

\smallskip

This variant is not only still sufficient for h-initiality, but also necessary, as we state precisely in the following.

\begin{proposition} \label{prop:2hinitial}
Over the type theory $\Hint$, the formation, introduction, elimination, and propositional computation rules
for $\Bool$ are equivalent to the existence of a homotopy-initial $\Bool$-algebra.
\end{proposition}
\begin{proof}[Proof sketch] 
Suppose we have a type $\Bool$ satisfying the stated rules.  Then clearly $(\Bool, 0, 1)$ is a $\Bool$-algebra; to show that it is h-initial, take any $\Bool$-algebra $(C,c_0,c_1)$.  By elimination with respect to the constant family $C$ and the elements $c_0$ and $c_1$, we have the map $\lambda x . \boolrec(x, c_0,c_1) : \Bool \rightarrow C$, which is a weak algebra homomorphism by the propositional computation rules.  Thus we obtain a term $h:\BoolAlg\big[ (\Bool, 0, 1), (C, c_0,c_1)\big]$.  Now given any $k:\BoolAlg\big[ (\Bool, 0, 1), (C, c_0,c_1)\big]$, we need a term of type $\id{}(h,k)$.  This term follows from a propositional $\eta$-rule, which is derivable by $\Bool$-elimination over a suitable identity type.

Conversely, let $(\Bool, 0, 1)$ be an h-initial $\Bool$-algebra.  To prove elimination, let $x:\Bool \vdash C(x):\type$ with $c_0 : C(0)$ and $ c_1 : C(1)$ be given, and consider the $\Bool$-algebra $(C', c'_0, c'_1)$ defined by:
\begin{align*}
C' &\defeq (\Sigma x: \Bool)C(x) \, , \\
c'_0 &\defeq \pair(0, c_0) \, , \\
c'_1 &\defeq \pair(1, c_1)\, .
\end{align*}
Since $\Bool$ is h-initial, there is a map $r : \Bool\rightarrow C'$ with identities $p_0:\id{}(r  0, c'_0)$ and $p_1:\id{}(r  1, c'_1)$.  Now, we would like to set 
$$\boolrec(x, c_0, c_1) = \pi_2 (r x) : C(x),$$
 where $\pi_2$ is the second projection from $C'=(\Sigma x: \Bool)C(x)$.  But recall that in general 
 $\pi_2(z) : C(\pi_1(z))$, and so (taking the case $x=0$) we have $\pi_2(r   0) : C(\pi_1(r   0))$ rather than the required $\pi_2(r   0) {\, : \, } C(0)$; that is, since it need not be that $\pi_1(r   0) = 0$, the term $\pi_2(r   0)$ has the wrong type to be $\boolrec(0, c_0, c_1)$.  However, we can show that $$\pi_1: (\Sigma x: \Bool)C(x)\rightarrow \Bool$$ is a weak homomorphism, so that the composite $\pi_1\circ r : (\Bool, 0, 1)\rightarrow (\Bool, 0, 1)$ must be propositionally equal to the identity homomorphism $1_\Bool : (\Bool, 0, 1)\rightarrow (\Bool, 0, 1)$, by the contractibility of $\BoolAlg \big[ (\Bool, 0, 1), (\Bool, 0, 1)\big]$.  Thus there is an identity term $p : \id{}(\pi_1\circ r, 1_\Bool)$, along which we can transport using $p_! : C(\pi_1( r   0)) \rightarrow C(0)$, thus taking $\pi_2(r   0 ) : C(\pi_1 (r   0))$ to the term  $p_! ( \pi_2( r   0) ) :C(0)$ of the correct type.  We can then set
\[
\boolrec(x, c_0, c_1) = p_! (\pi_2 (r x)) : C(x)
\]
to get the required elimination term.  The computation rules follow by a rather lengthy calculation.
\end{proof}

Proposition~\ref{prop:2hinitial} is the analogue in Homotopy Type Theory of the characterization of 
$\Bool$ as a strict coproduct $1+1$ in extensional type theory. It makes precise the rough idea that, 
in intensional type theory, $\Bool$ is a kind of homotopy coproduct or weak $\omega$-coproduct 
in the weak $\omega$-category $\mathcal{C}(\Hint)$ of types, terms, identity terms, higher identity terms, \ldots.  
It is worth emphasizing that h-initiality is a purely type-theoretic notion; despite having an obvious semantic interpretation, it is formulated in terms of inhabitation of specific, definable types.  Indeed, Proposition~\ref{prop:2hinitial} and its proof have been completely formalized in the Coq proof assistant~\cite{AwodeyS:indtht}.  

\begin{remark} A development entirely analogous to the foregoing can be given for the type
$\nat$ of natural numbers. In somewhat more detail, one introduces the notions of a $\nat$-algebra and 
of a weak homomorphism of $\nat$-algebras. Using these, it is possible to define the notion of
a homotopy-initial $\nat$-algebra, analogue to that of a homotopy-initial
$\Bool$-algebra in Definition~\ref{thm:boolhinitial}. With these definitions in place, one can prove an equivalence between the formation, introduction, elimination and propositional computation rules for $\nat$ and the existence of a homotopy-initial $\nat$-algebra. 
Here, the propositional computation rules are formulated like those above, \emph{i.e.} by replacing the definitional equalities in the conclusion of the usual computation rules~\cite[Section~5.3]{NordstromB:marltt} with propositional equalities. 
 We do not pursue this further here, however, since $\nat$ can also be presented as a W-type, as we discuss in section \ref{subsec:define} below.
\end{remark}

\subsection{The main theorem}\label{subsection:main}

\noindent Although it is more elaborate to state (and difficult to prove) owing to the presence of 
recursively generated data, our main result on  W-types is analogous to 
the foregoing example in the following respect: rather than being strict initial algebras, as in the 
extensional case, weak W-types are instead homotopy-initial algebras.  This fact can again be stated 
entirely syntactically, as an equivalence between two sets of rules:  the 
formation, introduction, elimination, and propositional computation rules (which we spell
out below) for W-types, and the existence of an h-initial algebra, in the appropriate sense.  Moreover, as in the simple case of the type $\Bool$, the proof of the equivalence is again entirely constructive. 

The required definitions in the current setting are as follows. Let us assume that
\[
x:A \vdash B(x) : \type \, ,
\]
and define the associated polynomial functor as before: 
\begin{equation}
\label{eq:polyfunc}
PX = (\Sigma x : A) (B(x) \rightarrow X) \, .
\end{equation}
(Actually, this is now functorial only up to propositional equality, but this change makes no difference in what follows.)
By definition, a $P$-algebra is a type $C$ equipped a function
$s_C :  PC \rightarrow C$. For $P$-algebras $(C,s_C)$ and $(D,s_D)$, a \emph{weak 
homomorphism} between them $(f, s_f) : (C, s_C) \rightarrow (D, s_D)$
consists of a function $f : C \rightarrow D$ and an identity proof
\[
s_f : \id{PC \rightarrow D}\big( f \circ s_C \, ,  s_{D} \circ Pf \big) \, ,
\]
where $Pf : PC\rightarrow PD$ is the result of the easily-definable action of $P$ on $f: C \rightarrow D$. Such an algebra homomorphism can be represented suggestively in the form:
\[
\xymatrix{
 PC \ar[d]_{s_C} \ar[r]^{Pf}  \ar@{}[dr]|{s_f} &  PD \ar[d]^{s_D}\\
C \ar[r]_{f}   & D }
\] 
Accordingly, the type of weak algebra maps is defined by
\begin{multline*}
\Palg
\big[ (C,s_C), (D, s_D)  \big]
 \defeq   \\
(\Sigma f:  C \rightarrow D) \, \Id(f\circ s_C, s_D\circ Pf) \, .
\end{multline*}

\begin{definition} 
A $P$-algebra $(C, s_C)$ is \emph{homotopy-initial} if for every $P$-algebra $(D, s_D)$, the type 
$$\Palg \big[ (C, s_C), (D, s_D) \big]$$ of weak algebra maps is contractible.
\end{definition}

\begin{remark} 
The notion of h-initiality captures a universal property in which the usual conditions of existence and uniqueness  are replaced by conditions of existence and uniqueness up to a system of higher and higher identity proofs. To explain this, let us fix a $P$-algebra  $(C,s_C)$ and assume that it is homotopy-initial. Then, given any 
$P$-algebra~$(D,s_D)$, there is a weak homomorphism $(f,s_f) : (C,s_C) \rightarrow (D,s_D)$, since
the type of weak maps from $(C,s_C)$ to $(D,s_D)$, being contractible, is inhabited. Furthermore, for any weak map $(g,s_g) : (C,s_C) \rightarrow (D,s_D)$, the
contractibility of the type of weak maps  implies that there is an identity proof 
\[
 p  : \Id\big( (f,s_f), (g, s_g) \big) \, , 
\]
witnessing the uniqueness up to propositional equality of the homomorphism $(f,s_f)$. But it
is also possible to prove that the identity proof $p$ is unique up to propositional equality. Indeed, since 
$(f,s_f)$ and $(g,s_g)$ are elements of a contractible type, the identity type $\Id( (f,s_f), (g, s_g) )$ 
is also contractible, as observed in Remark~\ref{thm:idcontrcontr}. Thus, if we have another 
identity proof $q : \Id( (f,s_f), (g, s_g) )$, there will be an identity term $\alpha : \Id(p,q)$, which is again
essentially unique, and so on.  It should also be pointed out that, just as strictly initial algebras are unique
up to isomorphism, h-initial algebras are unique up to weak equivalence. It then follows from
the Univalence axiom that two h-initial algebras are propositionally equal, a fact that we mention only by the way. Finally, we note that there is also a homotopical version of Lambek's Lemma, asserting that the structure map of an h-initial algebra is itself a weak equivalence, making the algebra a \emph{homotopy fixed point} of the associated polynomial functor. The reader can work out the details from the usual proof and the definition of  h-initiality, or consult~\cite{AwodeyS:indtht}.
\end{remark}

\medskip

The deduction rules that characterize homotopy-initial algebras are
obtained from the formation, introduction, elimination and computation rules for W-types 
stated in Section~\ref{section:extW} by simply replacing the $\W$-computation rule with the
following rule, that we call the propositional $\W$-computation rule.

\smallskip

\begin{itemize} 
\item Propositional $\W$-computation rule.
\[
\begin{prooftree}
\begin{array}{l} 
w : W \vdash C(w) : \type \\ 
\hspace{-1ex}
\begin{array}{c} 
x : A, u : B(x) \rightarrow W, v : (\Pi y : B(x)) C(u(y)) \vdash \\ 
c(x,u,v) : C(\wsup(x,u))  
\end{array}
\end{array}
\justifies
\begin{array}{l} 
x : A, u : B(x) \rightarrow W \vdash 
\wcomp(x,u,c) :  \\
\qquad \Id
\big(
\wrec(\wsup(x,u), c), c(x,u,\lambda y.\wrec( u(y), c )
\big)
\end{array}
\end{prooftree}
\]
\end{itemize}

 \begin{remark}\label{thm:wtypesinvariance}
One interesting aspect of this group of rules, to which we shall refer as the \emph{rules for homotopical
W-types}, is that, unlike the standard rules for W-types, they are invariant under propositional
equality. To explain this  more precisely, let us work in a type theory with a type universe $\UU$ closed
under all the forms of types of $\Hint$ and W-types. Let $A : \UU$, $B : A
\rightarrow \UU$ and define $W \defeq (\W x : A) B(x)$. The invariance of
the rules for homotopy W-types under propositional equality can now be expressed by saying that if we have a type $W' : \UU$  and an identity proof $p : \id{U}(W, W')$, then the $\Id$-elimination rule 
implies that  $W'$ satisfies the same rules as $W$, in the sense that there are definable terms playing the role of the primitive  constants that
appear in the rules for $W$.
\end{remark}

We can now state our main result. Its proof has been formalized in the Coq system,
and the proof scripts are available at~\cite{AwodeyS:indtht}; thus we provide only an outline of the proof. 

\begin{theorem}\label{theorem:main}
Over the type theory $\Hint$, 
the rules for homotopical W-types
are equivalent to the existence of homotopy-initial algebras for polynomial functors.
\end{theorem}

\begin{proof}[Proof sketch] 
The two implications are proved separately.
First, we show that the rules for homotopical W-types imply the existence
of homotopy-initial algebras for polynomial functors. Let us assume that
$x : A \vdash B(x) : \type$ 
and consider the associated polynomial functor $P$, defined as in~\eqref{eq:polyfunc}.
Using the $\W$-formation rule, we define $W \defeq (\W x : A) B(x)$ and using
the $\W$-introduction rule we define a structure map $s_W : PW \rightarrow W$,
exactly as in the extensional theory. We claim that the algebra $(W, s_W)$ is
h-initial. So, let us consider another algebra $(C,s_C)$ and prove that the type $T$ 
of weak homomorphisms from $(W, s_W)$ to $(C,s_C)$ is contractible. To do
so, observe that the $\W$-elimination rule and the propositional $\W$-computation
rule allow us to define a weak homomorphism $(f, s_f) : (W, s_W) \rightarrow (C, s_C)$, 
thus showing that $T$ is inhabited. Finally, it is necessary to show that for every weak homomorphism $(g, s_g) : (W, s_W) \rightarrow (C, s_C)$, there
is an identity proof 
\begin{equation}
\label{equ:prequired}
p : \Id( (f,s_f), (g,s_g) ) \, .
\end{equation}
This uses the fact that, in general, a type of the form $\Id( (f,s_f), (g,s_g) )$,
is weakly equivalent to the type of what we call \emph{algebra $2$-cells}, whose canonical
elements are pairs of the form $(e, s_e)$, where $e : \Id(f,g)$ and $s_e$ is a higher identity proof witnessing the propositional equality between the identity proofs represented by the following pasting diagrams:
\[
\xymatrix{
PW \ar@/^1pc/[r]^{Pg}   \ar[d]_{s_W} \ar@{}[r]_(.52){s_g}  & PD \ar[d]^{s_D}  \\
W \ar@/^1pc/[r]^g  \ar@/_1pc/[r]_f  \ar@{}[r]|{e} & D } \qquad
\xymatrix{
PW \ar@/^1pc/[r]^{Pg}   \ar[d]_{s_W}   \ar@/_1pc/[r]_{Pf} \ar@{}[r]|{Pe}
& PD \ar[d]^{s_D}  \\
W  \ar@/_1pc/[r]_f  \ar@{}[r]^{s_f} & D }
\]
In light of this fact, to prove that there exists a term as in~\eqref{equ:prequired}, it is 
sufficient to show that there is an algebra 2-cell 
\[
(e,s_e) : (f,s_f) \Rightarrow (g,s_g) \, .
\]
The identity proof $e : \Id(f,g)$ is now constructed by function extensionality and
$\W$-elimination so as to guarantee the existence of the required identity
proof $s_e$. 

\smallskip

For the converse implication, let us assume that the polynomial functor associated
to the judgement $x : A \vdash B(x) : \type$ 
has an h-initial algebra $(W,s_W)$. To derive the $\W$-formation rule, we 
let  $(\W x {\, : \, } A) B(x) \defeq W$. The $\W$-introduction rule is equally simple to
derive; namely, for $a : A$ and $t \colon B(a) \rightarrow W$,  we define $\wsup(a,t) : W$ as the 
result of applying the structure map $s_W \colon PW \rightarrow W$ to $\pair(a,t) : PW$.
For the $\W$-elimination rule, let us assume its premisses and in particular that $w : W \vdash C(w) : \type$.
Using the other premisses, one shows that the type $C \defeq (\Sigma w : W) C(w)$
can be equipped with a structure map $s_C : PC \rightarrow C$. By the h-initiality of $W$,
we obtain a weak homomorphism $(f, s_f) : (W, s_W) \rightarrow (C, s_C)$. Furthermore,
the first projection $\pi_1 : C \rightarrow W$ can be equipped with the structure of a weak 
homomorphism, so that we obtain a diagram of the form
\[
\xymatrix{
PW \ar[r]^{Pf} \ar[d]_{s_W}  & PC \ar[d]^{s_C}  \ar[r]^{P \pi_1}  & PW  \ar[d]^{s_W}  \\
W \ar[r]_f  & C \ar[r]_{\pi_1}  & W \, .}
\]
But the identity function $1_W : W \rightarrow W$ has a canonical structure of a weak
algebra homomorphism and so, by the contractibility of the type of weak homorphisms
from $(W,s_W)$ to itself, there must be an identity proof between the composite
of $(f,s_f)$ with $(\pi_1, s_{\pi_1})$ and $(1_W, s_{1_W})$. This implies, in particular,
that there is an identity proof $p : \Id( \pi_1 \circ f, 1_W)$. 
Since $(\pi_2 \circ f) w : C( (\pi_1 \circ f) w)$, we can define
\[
\wrec(w,c) \defeq
p_{\, ! \,}( ( \pi_2 \circ  f)   w )   : C(w) 
\]
where the transport $p_{\, ! \,}$ is defined via $\Id$-elimination over the dependent type
\[
u : W \rightarrow W \vdash C ( u (w)) : \type \, .
\]
The verification of the propositional $\W$-computation rule is a rather long calculation,
involving several lemmas concerning the naturality properties of operations of the form $p_{\, ! \,}$.
\end{proof}

\subsection{Definability of inductive types}\label{subsec:define}

\noindent We conclude this section by indicating how the limited form of extensionality that is assumed in the type theory~$\Hint$, namely the principle of function extensionality, allows us to overcome the obstacles in defining various inductive types as W-types mentioned at the end of Section~\ref{section:extW}, provided that both are understood in the appropriate homotopical way, \emph{i.e.} with all types being formulated with propositional computation rules. 

Consider first the paradigmatic case of the type of natural numbers. To define it as a W-type, we work in an extension of the type theory $\Hint$ with
\begin{itemize}
\item formation, introduction, elimination and propositional computation rules 
for types $\mathsf{0}$, $\mathsf{1}$ and $\mathsf{2}$ that have zero, one and two canonical elements,
respectively; 
\item the rules for homotopy W-types, as stated above;
\item rules for a type universe $\UU$ reflecting all the forms of types of $\Hint$, W-types, and
$\mathsf{0}$, $\mathsf{1}$ and $\mathsf{2}$.
\end{itemize}
In particular, the rules for $\mathsf{2}$ are those given in Section~\ref{subsection:prelimex}.
We then proceed as follows. We begin by setting $A =\mathsf{2}$, as in the extensional case.  
We then define a dependent type
\[
x : A \vdash B(x) : \UU
\]
by $\mathsf{2}$-elimination, so that the propositional $\mathsf{2}$-computation rules give us 
propositional equalities
\[
 p_0 : \id{U}( \mathsf{0}, B(0)) \, , \qquad
 p_1 : \id{U}( \mathsf{1}, B(1)) \, .
\]
Because of the invariance of the rules for $\mathsf{0}$ and $\mathsf{1}$ under propositional
equalities (as observed in Remark~\ref{thm:wtypesinvariance}), we can then derive that the types~$B(0)$ and~$B(1)$ satisfy rules analogous to those for $\mathsf{0}$ and $\mathsf{1}$, respectively. This allows us to show that the type 
\[
\nat \defeq (\W x : A) B(x)
\]
satisfies the introduction, elimination and propositional computation rules for the type of natural numbers. 
The proof of this fact proceeds essentially as one would expect, but to derive the propositional computation rules it
is useful to observe that for every type $X : \UU$, there are adjoint homotopy equivalences,
in the sense of Definition \ref{thm:weq}, between the types $\mathsf{0}  \rightarrow X$ and $\mathsf{1}$, 
and between $\mathsf{1} \rightarrow X$ and $X$. Indeed, the propositional identities witnessing 
the triangular laws are useful in the verification of the propositional computation rules for $\nat$. For 
details, see the formal development in Coq provided in~\cite{AwodeyS:indtht}.  Observe that as a W-type, $\nat$ is therefore also an h-initial algebra for the equivalent polynomial functor $P(X) = \mathsf{1}+ X$, as expected.

Finally, let us observe that the definition of a type representing the second number class as a W-type,
as discussed in~\cite{MartinLofP:inttt}, carries over equally well. Indeed, one
now must represent type-theoretically a signature with three operations: the first of arity zero, 
the second of arity one, and the third of arity $\nat$. For the first two we can proceed exactly as
before, while for the third there is no need to prove auxiliary results on adjoint homotopy 
equivalences. As before, the second number class supports an h-initial algebra structure for the corresponding polynomial functor $P(X) = \mathsf{1} + X + (\nat \rightarrow X)$.
Again, the formal development of this result in Coq can be found in~\cite{AwodeyS:indtht}.


\section{Future work}\label{section:future}

\noindent The treatment of W-types presented here is part of a larger investigation of general inductive types in Homotopy Type Theory.  We sketch the projected course of our further research.

\begin{enumerate}
\item  In the setting of  extensional type theory, Dybjer \cite{DybjerP:repids} showed that every strictly positive definable functor can be represented as a polynomial functor, so that all such inductive types are in fact W-types.  This result should generalize to the present setting in a straightforward way.

\item Also in the extensional setting, Gambino and Hyland~\cite{GambinoN:weltdp} showed that 
general tree types~\cite{PeterssonK:setcis}~\cite[Chapter~16]{NordstromB:promlt}, viewed as 
initial algebras for general polynomial functors, can be constructed from W-types in
locally cartesian closed categories, using equalizers. We expect this result to carry over to the present setting as well, using $\id{}$-types in place of equalizers.

\item In \cite{VoevodskyV:notts} Voevodsky has shown that all inductive types of the Predicative Calculus of Inductive Constructions can be reduced to the following special cases:
\begin{itemize}
\item $\mathsf{0}$,\ $\mathsf{1}$,\ $A+B$,\  $(\Sigma x : A)B(x)$,
\item $\id{A}(a,b)$,
\item general tree types.
\end{itemize}
Combining this with the foregoing, we expect to be able to extend our Theorem \ref{theorem:main} to the full system of predicative inductive types underlying Coq.
\end{enumerate}

\noindent Finally, one of the most exciting recent developments in Univalent Foundations is the 
idea of Higher Inductive Types (HITs), which can also involve identity terms in their signature~\cite{LumsdaineP:higit,ShulmanM:higit}.   This allows for algebras with equations between terms, like associative laws, coherence laws, etc.; but the really exciting aspect of HITs comes from the homotopical interpretation of identity terms as paths.  Viewed thus, HITs should permit direct formalization of many basic geometric spaces and constructions, such as the unit interval $I$; the spheres $S^n$, tori, and cell complexes; truncations, such as the [bracket] types \cite{AwodeyS:prot}; various kinds of quotient types; homotopy (co)limits; and many more fundamental and fascinating objects of geometry not previously captured by type-theoretic formalizations.  Our investigation of conventional inductive types in the homotopical setting should lead to a deeper understanding of these new and important geometric analogues.

\section*{Acknowledgements}

\noindent We would like to thank Andrej Bauer, Frank Pfenning, Robert Harper,
Vladimir Voevodsky and Michael Warren for helpful discussions
on the subject of this paper. In particular, Vladimir Voevodsky suggested a simplification of the proof that the rules for homotopical W-types imply h-initiality. 

Steve Awodey gratefully acknowledges the support of the National Science Foundation, Grant DMS-1001191
 and the Air Force OSR, Grant 11NL035.
Nicola Gambino is grateful for the support and the hospitality 
of the Institute for Advanced Study, where
he worked on this project. This work was supported by the National Science Foundation 
under agreement No.\ DMS-0635607. Any opinions, findings and conclusions or recommendations
expressed in this material are those of the authors and do not necessarily reflect the views of
the National Science Foundation.
Kristina Sojakova is grateful for the support of CyLab at Carnegie Mellon under grants DAAD19-02-1-0389
and W911NF-09-1-0273 from the Army Research Office, as well as for the support of the Qatar National Research Fund under grant NPRP 09-1107-1-168.



\begin{thebibliography}{40}


\bibitem{MartinLofP:intttp}
P.~Martin-L{\"o}f, ``An {I}ntuitionistic {T}heory of {T}ypes: {P}redicative
  {P}art,'' in \emph{Logic Colloquium 1973}, H.~Rose and J.~Shepherdson,
  Eds.\hskip 1em plus 0.5em minus 0.4em\relax North-Holland, 1975, pp. 73--118.

\bibitem{MartinLofP:conmcp}
P.~Martin-{L}\"of, ``Constructive mathematics and computer programming,'' in
  \emph{Proceedings of the Sixth International Congress for Logic, Methodology
  and Philosophy of Science}.\hskip 1em plus 0.5em minus 0.4em\relax
  North-Holland, 1982, pp. 153--175.

\bibitem{MartinLofP:inttt}
P.~Martin-L{\"o}f, \emph{Intuitionistic Type Theory. Notes by G. Sambin of a
  series of lectures given in Padua, 1980}.\hskip 1em plus 0.5em minus
  0.4em\relax Bibliopolis, 1984.

\bibitem{NordstromB:promlt}
B.~Nordstrom, K.~Petersson, and J.~Smith, \emph{Programming in {M}artin-{L}\"of
  type theory}.\hskip 1em plus 0.5em minus 0.4em\relax Oxford University Press,
  1990.

\bibitem{NordstromB:marltt}
------, ``{M}artin-{L}\"of type theory,'' in \emph{Handbook of Logic in
  Computer Science}.\hskip 1em plus 0.5em minus 0.4em\relax Oxford University
  Press, 2000, vol.~5, pp. 1--37.

\bibitem{HowardWH:foratn}
W.~H. Howard, ``The formulae-as-types notion of construction,'' in \emph{To
  {H}.{B}. {C}urry: {E}ssays on {C}ombinatory {L}ogic, {L}ambda {C}alculus and
  {F}ormalism}, J.~P. Seldin and J.~R. Hindley, Eds.\hskip 1em plus 0.5em minus
  0.4em\relax Academic Press, 1980, pp. 479--490.

\bibitem{GrifforE:strsml}
E.~Griffor and M.~Rathjen, ``The strength of some {M}artin-{L}\"of type
  theories,'' \emph{Archive for Mathematical Logic}, vol.~33, no.~5, pp. 347--385, 1994.

\bibitem{HofmannM:extcit}
M.~Hofmann, \emph{Extensional constructs in intensional type theory}.\hskip 1em
  plus 0.5em minus 0.4em\relax Springer-Verlag, 1997.

\bibitem{MaiettiME:mintlf}
M.~E. Maietti, ``A minimalist two-level foundation for constructive
  mathematics,'' \emph{Annals of Pure and Applied Logic}, vol. 160, no.~3, pp.
  319--354, 2009.

\bibitem{DybjerP:repids}
P.~Dybjer, ``Representing inductively defined sets by wellorderings in
  {M}artin-{L}{\"o}f's type theory,'' \emph{Theoretical Computer Science}, vol. 176, pp.
  329--335, 1997.

\bibitem{MoerdijkI:weltc}
I.~Moerdijk and E.~Palmgren, ``Wellfounded trees in categories,'' \emph{Annals
  of Pure and Applied Logic}, vol. 104, pp. 189--218, 2000.

\bibitem{GambinoN:weltdp}
N.~Gambino and M.~Hyland, ``{Wellfounded Trees and Dependent Polynomial
  Functors},'' in \emph{Types for Proofs and Programs (TYPES 2003)}, ser. LNCS,
  S.~Berardi, M.~Coppo, and F.~Damiani, Eds., vol. 3085, 2004, pp. 210--225.

\bibitem{AbbottM:concsp}
M.~Abbott, T.~Altenkirch, and N.~Ghani, ``Containers: Constructing strictly
  positive types,'' \emph{Theoretical Computer Science}, vol. 342, no.~1, pp. 3--27, 2005.

\bibitem{GoguenH:inddtw}
H.~Goguen and Z.~Luo, ``Inductive data types: well-ordering types revisited,''
  in \emph{Logical Environments}, G.~Huet and G.~Plotkin, Eds.\hskip 1em plus
  0.5em minus 0.4em\relax Cambridge University Press, 1993, pp. 198--218.

\bibitem{AwodeyS:homtmi}
S.~Awodey and M.~A. Warren, ``Homotopy theoretic models of identity types,''
  \emph{Mathematical Proceedings of the Cambridge Philosophical Society}, vol.
  146, pp. 45--55, 2009.

\bibitem{VoevodskyV:notts}
V.~Voevodsky, ``Notes on type systems,'' 2009, available from the author's web
  page.

\bibitem{vandenBergB:topsmi}
B.~van~den Berg and R.~Garner, ``Topological and simplicial models of identity
  types,'' 2011, arXiv:1007.4638v2. To appear in \emph{ACM Transactions in
  Computational Logic}.

\bibitem{VoevodskyV:unifp}
V.~Voevodsky, ``Univalent foundations project,'' 2010, available from the
  author's web page.

\bibitem{AwodeyS:indtht}
S.~Awodey, N.~Gambino, and K.~Sojakova, ``Inductive types in {H}omotopy {T}ype
  {T}heory: {C}oq proofs,'' 2012, available from 
\url{https://github.com/HoTT/Archive}.

\bibitem{AwodeyS:typth}
S.~Awodey, ``Type theory and homotopy,'' 2010, available from the author's web
  page.

\bibitem{BertotY:inttpp}
Y.~Bertot and P.~Cast{\'e}ran, \emph{Interactive Theorem Proving and Program
  Development. {C}oq'{A}rt: the {C}alculus of {I}nductive
  {C}onstructions}.\hskip 1em plus 0.5em minus 0.4em\relax Springer Verlag,
  2004.

\bibitem{NorellU:towppl}
U.~Norell, ``Towards a practical programming language based on dependent type
  theory,'' Ph.D. dissertation, Chalmers University of Technology, 2007.

\bibitem{McBrideC:viefl}
C.~McBride and J.~McKinna, ``The view from the left,'' \emph{Journal of
  Functional Programming}, vol.~14, no.~1, pp. 69--111, 2004.

\bibitem{CoquandT:inddt}
T.~Coquand and C.~Paulin-Mohring, ``Inductively defined types,'' in
  \emph{Proceedings of Colog'88}, ser. LNCS, vol. 417.\hskip 1em plus 0.5em
  minus 0.4em\relax Springer, 1990.

\bibitem{PaulinMorhringC:inddsc}
C.~Paulin-Mohring, ``Inductive definitions in the system {C}oq - {R}ules and
  {P}roperties,'' in \emph{Typed Lambda Calculi and Applications}, ser. LNCS,
  vol. 664.\hskip 1em plus 0.5em minus 0.4em\relax Springer, 1993.

\bibitem{McBrideC:wtygnb}
C.~McBride, ``W-types: good news and bad news,'' 2010, {P}ost on the {E}pigram
  blog.

\bibitem{StreicherT:invitt}
T.~Streicher, ``Investigations into intensional type theory,'' 1993,
  {H}abilitation Thesis. Available from the author's web page.

\bibitem{AltenkirchT:obsen}
T.~Altenkirch, C.~McBride, and W.~Swierstra, ``Observational equality, now!''
  in \emph{PLPV '07: Proceedings of the 2007 workshop on Programming languages
  meets program verification}.\hskip 1em plus 0.5em minus 0.4em\relax ACM,
  2007, pp. 57--68.

\bibitem{GarnerR:strdpt}
R.~Garner, ``On the strength of dependent products in the type theory of
  {M}artin-{L}\"of,'' \emph{Annals of Pure and Applied Logic}, vol. 160, pp. 1--12, 2009.

\bibitem{VoevodskyV:unifc}
V.~Voevodsky, ``Univalent foundations {C}oq files,'' 2010, available from the
  author's web page.

\bibitem{HofmannM:gromtt}
M.~Hofmann and T.~Streicher, ``The groupoid model of type theory,'' in
  \emph{Twenty-five years of constructive type theory}, G.~Sambin and J.~Smith,
  Eds.\hskip 1em plus 0.5em minus 0.4em\relax Oxford University Press, 1995.

\bibitem{BataninM:mongcn}
M.~Batanin, ``Monoidal globular categories as a natural environment for the
  theory of weak $n$-categories,'' \emph{Advances in Mathematics}, vol. 136,
  no.~1, pp. 39--103, 1998.

\bibitem{LeinsterT:higohc}
T.~Leinster, \emph{Higher operads, higher categories}.\hskip 1em plus 0.5em
  minus 0.4em\relax Cambridge University Press, 2004.

\bibitem{vandenBergB:typwg}
B.~van~den Berg and R.~Garner, ``Types are weak $\omega$-groupoids,''
  \emph{Proceedings of the London Mathematical Society}, vol. 102, no.~3, pp.
  370--394, 2011.

\bibitem{LumsdaineP:weaci}
P.~Lumsdaine, ``Weak $\omega$-categories from intensional type theory,'' in
  \emph{Typed Lambda Calculi and Applications}, ser. LNCS, P.-L. Curien, Ed.,
  no. 5608.\hskip 1em plus 0.5em minus 0.4em\relax Springer, 2009, pp.
  172--187.

\bibitem{Lumsdaine:higcft}
------, ``Higher categories from type theories,'' Ph.D. dissertation, Carnegie
  Mellon University, 2010.

\bibitem{PeterssonK:setcis}
K.~Petersson and D.~Synek, ``A set constructor for inductive sets in
  {M}artin-{L}\"of type theory,'' in \emph{Proceedings of the 1989 Conference
  on Category Theory and Computer Science, Manchester, {U.}{K.}}, ser. LNCS,
  vol. 389.\hskip 1em plus 0.5em minus 0.4em\relax Springer-Verlag, 1989.

\bibitem{LumsdaineP:higit}
P.~Lumsdaine, ``Higher inductive types: a tour of the managerie,'' 2011, {P}ost
  on the Homotopy Type Theory blog.

\bibitem{ShulmanM:higit}
M.~Shulman, ``Homotopy {T}ype {T}heory, {VI},'' 2011, {P}ost on the
  $n$-category caf\'e blog.

\bibitem{AwodeyS:prot}
S.~Awodey and A.~Bauer, ``Propositions as [types],'' \emph{Journal of Logic and
  Computation}, vol.~14, no.~4, pp. 447--471, 2004.

\end{thebibliography}
\end{document}